\documentclass[openany, amssymb, psamsfonts]{amsart}
\usepackage{mathrsfs,comment}
\usepackage[usenames,dvipsnames]{color}
\usepackage[normalem]{ulem}
\usepackage{url}
\usepackage[all,arc,2cell]{xy}
\UseAllTwocells
\usepackage{enumerate}
\usepackage{hyperref}  
\hypersetup{%
  bookmarksnumbered=true,%
  bookmarks=true,%
  colorlinks=true,%
  linkcolor=blue,%
  citecolor=blue,%
  filecolor=blue,%
  menucolor=blue,%
  pagecolor=blue,%
  urlcolor=blue,%
  pdfnewwindow=true,%
  pdfstartview=FitBH}

\def\makeautorefname#1#2{\expandafter\def\csname#1autorefname\endcsname{#2}}
\def\equationautorefname~#1\null{(#1)\null}
\makeautorefname{footnote}{footnote}%
\makeautorefname{item}{item}%
\makeautorefname{figure}{Figure}%
\makeautorefname{table}{Table}%
\makeautorefname{part}{Part}%
\makeautorefname{appendix}{Appendix}%
\makeautorefname{chapter}{Chapter}%
\makeautorefname{section}{Section}%
\makeautorefname{subsection}{Section}%
\makeautorefname{subsubsection}{Section}%
\makeautorefname{theorem}{Theorem}%
\makeautorefname{thm}{Theorem}%
\makeautorefname{cor}{Corollary}%
\makeautorefname{lem}{Lemma}%
\makeautorefname{prop}{Proposition}%
\makeautorefname{pro}{Property}
\makeautorefname{conj}{Conjecture}%
\makeautorefname{defn}{Definition}%
\makeautorefname{notn}{Notation}
\makeautorefname{notns}{Notations}
\makeautorefname{rem}{Remark}%
\makeautorefname{quest}{Question}%
\makeautorefname{exmp}{Example}%
\makeautorefname{ax}{Axiom}%
\makeautorefname{claim}{Claim}%
\makeautorefname{ass}{Assumption}%
\makeautorefname{asss}{Assumptions}%
\makeautorefname{con}{Construction}%
\makeautorefname{prob}{Problem}%
\makeautorefname{warn}{Warning}%
\makeautorefname{obs}{Observation}%
\makeautorefname{conv}{Convention}%

\newtheorem{thm}{Theorem}[section]
\newtheorem{cor}{Corollary}[section]
\newtheorem{prop}{Proposition}[section]
\newtheorem{lem}{Lemma}[section]

\theoremstyle{definition}
\newtheorem{defn}{Definition}[section]

\newtheorem{quest}{Question}[section]
\newtheorem{rem}{Remark}[section]

\makeatletter
\let\c@obs=\c@thm
\let\c@cor=\c@thm
\let\c@prop=\c@thm
\let\c@lem=\c@thm
\let\c@prob=\c@thm
\let\c@con=\c@thm
\let\c@conj=\c@thm
\let\c@defn=\c@thm
\let\c@notn=\c@thm
\let\c@notns=\c@thm
\let\c@exmp=\c@thm
\let\c@ax=\c@thm
\let\c@pro=\c@thm
\let\c@ass=\c@thm
\let\c@warn=\c@thm
\let\c@rem=\c@thm
\let\c@sch=\c@thm
\let\c@equation\c@thm
\numberwithin{equation}{section}
\makeatother

\title[Topological Obstructions] {Topological Obstructions to the Existence of Compact Shrinking Ricci Solitons in Dimension Four}

\author{Cameron MacMahon}

\begin{document}

\begin{abstract}

This undergraduate thesis is focused on introducing the reader to concepts related to the search for topological obstructions to the existence of compact gradient shrinking Ricci soliton metrics in dimension four. It contains a discussion of the relevant background material for this subject. Furthermore, it introduces the problem of extending the Hitchin-Thorpe inequality to gradient shrinking Ricci soliton metrics and explores the limitations of current results in that direction. At last, the topic of compact Kähler gradient shrinking Ricci solitons is introduced and the classification of these spaces is outlined in literature-study fashion. 

\end{abstract}

\maketitle

\tableofcontents

\section{Editor's Note}

Brian Klatt has informed the author of two mistakes concerning Theorem 5.13, Theorem 5.14, and the remark following it. These change not only the correctness of those statements themselves but the philosophical conclusion the author initially meant to draw at the end of section 5. The reader should refer back to this section as they read section 5, as the corrections have not been incorporated into the text (for the sake of the author's preserving his undergraduate thesis as it had been written). Klatt's remarks are below. The first concerns Theorem 5.13 below. The result in \cite{chang} only guarantees that there is a metric conformally equivalent to any metric in $\mathcal{A}$ with positive Ricci curvature. This invalidates the proof of Theorems 5.13 and thus 5.14. The second concerns the remark following Theorem 5.14 that many examples of Ricci solitons are $``$missed" by the sufficient conditions in Theorems 5.5 and 5.6. The Koiso-Cao soliton in complex dimension 2 (our case) has positive Ricci curvature, therefore it is not a-priori $``$missed" by the sufficient conditions in Theorems 5.5 and 5.6. In fact it appears to satisfy both! As Klatt points out, the expression for the Euler characteristic on page 9 used in this paper differs from Chang's \cite{chang} conventions in that there is a factor of four missing from the $|W|^2$ terms. Taking into account the difference in convention and computing the integral of $\sigma_2$ over the Koiso-Cao soliton in fact yields a positive number. Theorems proven in this paper then allow us to conclude that the Koiso-Cao soliton actually satisfies the sufficient conditions in Theorems 5.5 and 5.6 meaning it is not $``$missed" as initially claimed. The author thanks Brian Klatt for bringing these mistakes to his attention. 

\section{Introduction}

    The purpose of this thesis is to introduce the reader to the current body of knowledge surrounding the search for topological obstructions to the existence of Gradient Shrinking Ricci Soliton metrics on $\emph{compact}$ four-manifolds, perhaps with some additional structure included (e.g. Kähler). These are metrics $g$ which satisfy 
    
    $$Rc + \nabla^2f = \rho g$$
    
    For some function $f\in C^\infty(\mathcal{M})$ and $\rho > 0$ where $Rc$ is the Ricci tensor of $g$. The serious study of compact Gradient Shrinking Ricci Soliton metrics goes back at least as far as Hamilton's initiation of a program which uses the Ricci flow to study the geometry and topology of Riemannian Manifolds. The Ricci flow is the following differential equation in the metric tensor which, very roughly speaking evolves the geometry in the opposite direction of its curvature:
    
    $$\frac{\partial g(t)}{\partial t} = -2Rc$$
    
    Gradient shrinking Ricci solitons in general are important spaces because they can arise from the Ricci flow as $\emph{models of singularities}$ whose evolution is given by homothetic shrinking and evolution of the initial metric $g_0$ along integral curves of the gradient of $f$:
    
    $$g(t) = (T - t)\phi^*_{t}(g_0)$$
    
    An understanding of the singularities in the Ricci flow - and thus gradient shrinking Ricci solitons - is incredibly important because it leads to the technique of Ricci flow with surgery, permitting one to continue the flow after it has reached a singular configuration. Indeed, it was an understanding of gradient shrinking Ricci solitons in dimension three that enabled the proof of Thurston's geometrization conjecture of which the Poincare conjecture is a special case. In particular, much is known about compact gradient shrinking Ricci solitons in low dimensions other than $4$. In the case $n=2$, Hamilton has demonstrated that any such metric (in the compact case) must be isometric to a quotient of $S^2$ \cite{ham}. Similarly, in the case $n=3$ Perelman \cite{perl} and Ivey \cite{ivey} have demonstrated that compact gradient shrinking Ricci solitons are finite quotients of $S^3$. Therefore, in low dimensions the $n=4$ case represents the only severe gap in our knowledge. 

    One interesting way one can rule out the existence of compact gradient shrinking Ricci soliton metrics is to study necessary conditions on the topology which arise from the existence of such metrics. Since the non-satisfaction of such a necessary condition would preclude the existence of such a metric, these conditions are known as $\emph{obstructions}$.

    In what follows, two avenues will be taken to explore what restrictions exist within the topology of compact gradient shrinking Ricci solitons in dimension four. In section 2, compact gradient Ricci solitons are introduced in context, and we prove that consideration of the so-called $\emph{expanding}$ and $\emph{steady}$ cases amounts to an understanding of the topology of Einstein manifolds. In section 3, background for those parts of differential geometry and topology that are particular to four dimensions is provided and the local-to-global theorems which will furnish topological obstructions are introduced. In section 4 we investigate progress towards a particular topological obstruction to compact gradient shrinking Ricci soliton metrics which is natural to consider given the fact that compact gradient shinking Ricci solitons may be considered generalizations of compact positive Einstein manifolds. In particular, it is known that compact positive Einstein manifolds in dimension four must satisfy the $\emph{Hitchin-Thorpe inequality}$ relating their Euler number $\chi$ and their signature $\tau$ \cite{ernani}:

    $$2\chi \pm 3|\tau| \geq 0$$

    This is an obstruction to the existence of compact positive Einstein metrics because a compact four-manifold whose topology does not admit this relationship cannot have such a metric. Section 4 covers current progress in this direction, including proofs of two sufficient conditions that, if satisfied, guarantee that a given compact gradient shrinking Ricci soliton satisfies the Hitchin-Thorpe inequality. Section 4 also contains a discussion of the applicability of these sufficient conditions, and contains arguments using conformal geometry that restrict the class of manifolds which obey the sufficient conditions proved. Section 4 concludes with a brief meta-mathematical discussion about the status of the satisfaction of the Hitchin-Thorpe inequality for general compact gradient shrinking Ricci solitons.

    Section 5 gives an overview of the interplay between Kähler structure and gradient shrinking Ricci soltion structure on compact four-manifolds, and reads like a literature review. In particular, an obstruction akin to the Hitchin-Thorpe inequality is proved immediately. Section 5 then continues to outline a classification of all compact Kähler gradient shrinking Ricci solitons. A brief remark on extension of results to the symplectic case is then made. To that end, I would like to acknowledge my advisor, professor Huai-Dong Cao for introducing me to this area of research, for providing helpful answers to a multitude of questions, for directing my attention to the classification in section 5, and for taking me under his wing and developing me as an undergraduate mathematician. Similarly, I would like to thank professors Andrew Harder and Donald Davis, without whose tutelage and engaging problem sets I would not have been able to complete this thesis. I would also like to thank professor Claude LeBrun for a very helpful conversation, as well as for the useful background information on four-manifolds present in his articles. Lastly, I would like to thank professor Jenna Lay for being a constant source of advice, encouragement, and support throughout my undergraduate career. I credit the University of Chicago REU for their LaTeX format, available freely on their web-page.

\section{A Primer on Gradient Shrinking Ricci Solitons}

 Let $(\mathcal{M}^n, g)$ be a compact Riemannian manifold. Such a manifold is called a $\emph{compact gradient Ricci soliton}$ - or simply a gradient Ricci soliton as we work in the compact case from now on -  if the following holds for some $f \in C^\infty(\mathcal{M})$ and some $\rho \in \mathbb{R}$:

 $$Rc + \nabla^2f = \rho g$$

 where $Rc$ is the Ricci tensor of $g$. Note immediately that this equation is diffeomorphism invariant, scale invariant, and invariant under translation of the potential function. Similarly, note that when $f \equiv C$ for some constant $C$ (or just when the Hessian of $f$ vanishes), then the metric $g$ is an $\emph{Einstein metric}$. An $\emph{Einstein metric}$ is a Riemannian metric $g$ satisfying the following for some constant $\rho \in \mathbb{R}$: 
 
$$Rc = \rho g$$

Thus, gradient Ricci soliton metrics may be considered generalizations of Einstein metrics. Below, we demonstrate that the only nontrivial instances of this generalization occur when $\rho > 0$. Naturally, gradient Ricci solitons may be split into three types: $\emph{steady}$ solitons with $\rho = 0$, $\emph{expanding}$ solitons with $\rho < 0$, and $\emph{shrinking}$ solitons with $\rho > 0$. Our claim above now amounts to the statement that the only gradient Shrinking Ricci soliton metrics which are not Einstein metrics are of the shrinking type.

\begin{thm} (Hamilton, Ivey; see Cao-Zhu \cite{caoOne}) On a compact n-manifold $\mathcal{M}$, a steady or expanding gradient Ricci soliton is necessarily an Einstein metric. 

\end{thm}

\begin{proof}

Take the trace of the soliton equation, yielding $R + \Delta f = n\rho$. Furthermore, working in normal coordinates, let us take a covariant derivative of the soliton equation and subtract two copies of what remains, yielding

$$\nabla_i \nabla_j \nabla_k f - \nabla_j \nabla_i \nabla_k f = \nabla_j R_{ik} - \nabla_i R_{jk}.$$

We may recall also the Ricci identities for commuting covariant derivatives, yielding

$$\nabla_i \nabla_j \nabla_k f - \nabla_j \nabla_i \nabla_k f = R_{ijkl}\nabla_l f.$$

Combining, one has

$$\nabla_i R_{jk} - \nabla_j R_{ik} + R_{ijkl}\nabla_l f = 0.$$

Next, trace on $j$ and $k$ and recall the second contracted Bianchi identity $\nabla_i R_{ik} = \frac{1}{2} \nabla_k R$ to obtain

$$\frac{1}{2} \nabla_i R - R_{ij} \nabla_j f = 0.$$

We now consider the following quantity: $|\nabla f|^2 + R$. Taking a covariant derivative and applying the above equation we obtain (call this equation $(*)$):

$$\nabla_i(|\nabla f|^2 + R) = 2 \nabla_j f (\nabla_ i \nabla_j f + R_{ij}).$$

 We now proceed by cases. In the steady case, $\nabla_ i \nabla_j f + R_{ij} = 0$, and thus by $(*)$ above

 $$|\nabla f|^2 + R = C$$

for some constant $C \in \mathbb{R}$. Recalling the fact that $R + \Delta f = 0$ in the steady case we obtain

$$\Delta f - |\nabla f|^2 = -C.$$

We now prove that $C \equiv 0$ using the fact that $\mathcal{M}$ is compact (we may use Green's identities):

$$0 = \int_{\mathcal{M}} \Delta e^{-f} dV_g = \int_{\mathcal{M}}(\Delta f - |\nabla f|^2)e^{-f}dV_g =  \int_{\mathcal{M}} -C e^{-f}dV_g.$$

and thus $C \equiv 0$. We now have that $(\Delta f - |\nabla f|^2) = 0$ and thus - by applying Greens' identities again - we obtain

$$0 = \int_{\mathcal{M}} \Delta f dV_g = \int_{\mathcal{M}} |\nabla f|^2 dV_g$$

and therefore $f$ is a constant as $|\nabla f|^2$ is everywhere vanishing. Hence, $g$ is an Einstein metric.

We now apply the maximum principle to tackle the expanding case. Combining $(*)$ with the expanding soliton equation we get:

$$\nabla_i(|\nabla f|^2 + R) = 2 \nabla_j f (\rho g_{ij}) = 2\rho \nabla_i f$$

for some $\rho < 0$. Subtracting the right from the left and using linearity of $\nabla$, one can conclude for expanding solitons that 

$$\nabla_i(|\nabla f|^2 + R - 2\rho f) = 0$$

for some $\rho < 0$. Using the fact that $R + \Delta f = n\rho$, the potential function must satisfy

$$\Delta f - |\nabla f|^2 + 2\rho f = C$$

for some $\rho \leq 0$ and $C \in \mathbb{R}$. Given our freedom to add constants to the potential function and retain the Soliton structure, we may normalize: 

$$\Delta f - |\nabla f|^2 + 2\rho f = 0$$

We now apply the maximum principle. When $f$ attains a maximum, $\Delta f \leq 0$ and $\nabla f \equiv 0$, thus we have

$$\rho f_{max} \geq 0$$

and so $f$ must attain a maximum at a nonpositive value as $\rho < 0$. Furthermore, when $f$ attains a minimum, $\Delta f \geq 0$ and $\nabla f \equiv 0$, thus we have 

$$\rho f_{min} \leq 0$$

and so $f_{min}$ must be nonnegative since $\rho < 0$. Thus $f_{max} = f_{min} \equiv 0$. And therefore $f \equiv 0$ and thus $g$ is Einstein since the Hessian term in the Soliton equation vanishes. This concludes the proof.

\end{proof}

 In this paper, we are concerned with investigating topological obstructions to the existence of compact gradient Ricci solitons in dimension four (the dimension we will work in from here on out). There already exist obstructions such as the Hitchin-Thorpe inequality (to be discussed in a later section) which provide information about the non-existence of Einstein metrics on compact four-manifolds. Thus, given the above result, it is natural to ask whether or not these obstructions extend to give information about the non-existence of compact gradient $\emph{shrinking}$ Ricci solitons. One might naively assume that all compact gradient Shrinking Ricci solitons are Einstein as well, but this is in fact false. Koiso and Cao \cite{caoTwo} and Wang and Zhu \cite{wang} have constructed complete gradient shrinking Ricci solitons on $\mathbb{CP}^2 \# -(\mathbb{CP}^2)$ and $\mathbb{CP}^2 \# -2(\mathbb{CP}^2)$ which are not Einstein, demonstrating that we cannot carry the above argument any further. These solitons in fact turn out to be Kähler, and we demonstrate below in section $5$ that these are in fact the only compact simply connected smooth four-manifolds admitting non-Einstein Kähler gradient shrinking Ricci solitons. 

 Before moving further, however, we note two results crucial to our current understanding of compact gradient Shrinking Ricci solitons:

 \begin{thm} (Chen \cite{chen})
 On a compact gradient shrinking Ricci soliton, the scalar curvature $R > 0$. 

 \end{thm}

 Note that this implies that compact gradient shrinking Ricci solitons have positive Yamabe invariant. The following result was first proved in the compact gradient case by Lott \cite{lott}, and was actually proven in the general case of complete gradient shrinking Ricci solitons by Wylie \cite{wylie}:

 \begin{thm}(Lott \cite{lott})
 Compact gradient shrinking Ricci solitons have finite $\pi_1$.

 \end{thm}

 Note in particular that the first betti number $b_1$ of such a manifold must vanish. This is our first example of a topological obstruction to the existence of such a geometric structure. We now proceed to review the next ingredient in our study: topological invariants coming from aspects of differential topology peculiar to four dimensions.

 \section{Oriented Homotopy Invariants on Smooth Four-Manifolds}

    This section closely follows LeBrun \cite{lebrun}, and is included as a convenience to the reader. Be aware that any differences in notation or errors in this section - which I have done my best to avoid - are mine, and do not result from my drawing on this source for background knowledge. Differential geometry in four dimensions is different. This is because the real six-dimensional bundle $\Lambda^2(\mathcal{M}^4)$ of two forms on a four-dimensional smooth manifold decomposes into two real three-dimensional bundles:
    
    $$\Lambda^2(\mathcal{M}^4) = \Lambda^+ \oplus \Lambda^-$$

    The most concrete way to think of this decomposition involves the hodge star operator $*$. In dimension four, this operator satisfies $*^2 = 1$. Thus as a map $*: \Lambda^2(\mathcal{M}^4) \rightarrow \Lambda^2(\mathcal{M}^4)$ decomposes $\Lambda^2(\mathcal{M}^4)$ into $\pm 1$-eigenspaces of dimension three ($\Lambda^+$ and $\Lambda^-$ respectfully). 

    \begin{defn}
    Sections of $\Lambda^+$ will be called self-dual 2-forms and sections of $\Lambda^-$ will be called anti-self-dual 2-forms. 
    \end{defn}

    Now, let $\mathcal{M}^4$ be a compact oriented 4-manifold. Recall that one has a cup product in integral cohomology:

    $$\cup: H^2(\mathcal{M}^4; \mathbb{Z}) \times H^2(\mathcal{M}^4; \mathbb{Z})  \rightarrow H^4(\mathcal{M}^4; \mathbb{Z}) \cong \mathbb{Z}$$

    where the last isomorphism arises from Poincare duality and the fact that $\mathcal{M}^4$ is a compact, connected 4-manifold. Mod torsion, this bilinear pairing may be expressed as a $b_2$ by $b_2$ matrix $\tau$ of determinant $\pm 1$ with integer entries called the $\emph{intersection form}$ of $\mathcal{M}^4$. The name intersection form was coined with an eye towards the following concrete interpretation:

    \begin{thm}

    Let $(\mathcal{M}^4, g)$ be a smooth, compact, connected, oriented Riemannian 4-manifold. Let $\alpha, \beta \in H^2(\mathcal{M}^4; \mathbb{Z})$ and let $a, b$ be smoothly embedded oriented 2-manifolds representing the Poincare duals of $\alpha, \beta$. Assume without loss of generality that $a$ and $b$ intersect transversally at finitely many points. Now, $a$ and $b$ have orientations, and the direct sum of their tangent spaces at an intersection point is the tangent space of $\mathcal{M}^4$. For each intersection point, associate either a $+1$ or a $-1$ depending on if the orientation on $T_p\mathcal{M}^4$ inherited from the embedded surfaces by direct sum agrees or disagrees with the given orientation on $\mathcal{M}^4$ and sum these values. The sum will be $\tau(\alpha, \beta)$. 

    \end{thm}

    We may express these ideas in the language of forms if we're willing to move to $\mathbb{R}$-coefficients. The ability to work with forms is well worth the coefficient change. Recall that the Hodge theorem states that each de-Rham cohomology class $[\omega] \in H^2(\mathcal{M}^4;\mathbb{R})$ admits a harmonic representative, that is, a two-form $\phi$ representing the given cohomology class such that $d\phi = d *\phi = 0$. Recall also that, in the de-Rham cohomology, the cup product with $\mathbb{R}$ coefficients is simply the wedge product and the poincare-duality isomorphism is given by integrating top-forms over the manifold. Thus, we may consider the intersection pairing on harmonic 2-forms $\alpha, \beta$:

    $$\tau(\alpha, \beta) = \int_{\mathcal{M}} \alpha \wedge \beta$$

    Furthermore, $*$ descends to a map between harmonic two forms, since $$H^2(\mathcal{M}^2; \mathbb{R}) \cong \mathcal{H}^2(\mathcal{M}^4; \mathbb{R}) = \{\phi \in \Lambda^2(\mathcal{M}^4) : d\phi=0, d*\phi = 0\},$$

    and $*$ respects the set-membership condition above since $*^2=1$. Thus as above we have a decomposition $\mathcal{H} \cong \mathcal{H}^+ \oplus \mathcal{H}^-$ into self-dual and anti-self-dual harmonic 2-forms. Note that the harmonic condition may be relaxed to $d\phi=0$ in either summand because elements of each summand are eigenforms of $*$. Further note that $*$ is a conformal invariant in the middle dimension so the decomposition remains unchanged given a conformal scaling of the metric. $\tau$ becomes positive definite when restricted to $\mathcal{H}^+$ and negative-definite when restricted to $\mathcal{H}^-$, thus we may pick $L^2$-orthonormal bases for both subspaces so $\tau$ is expressed by $diag(+1,...,+1, -1,...,-1)$ with $b^+$ ones and $b^-$ minus ones. $b^\pm$ are oriented homotopy invariants of $M^4$. The $\emph{signature}$ of $\mathcal{M}^4$ is $b^+ - b^-$ and, by abuse of notation, is also denoted $\tau$. Like the Euler characteristic, this is an oriented homotopy invariant of $\mathcal{M}^4$. 

    In fact, $\chi$ and $\tau$ ($b^+$ and $b^-$) completely characterize the oriented homeotype of a compact, simply-connected 4-manifold: 

    \begin{thm}
    (Freedman \cite{freed}) Two smooth simply-connected, oriented, compact 4-manifolds are orientedly homeomorphic if and only if they have the same $\chi$ and $\tau$ (e.g. they have the same $b^+$ and $b^-$), and both are spin or both are non-spin.

    \end{thm}

    An oriented smooth Riemannian 4-manifold is called $\emph{spin}$ if the principle $SO(4)$ frame bundle of $\mathcal{M}^4$ admits a lifting to a principle $Spin(4)$ (the Lie-group double-cover of $SO(4)$) bundle such that modding out by the covering map fiber-wise yields the frame bundle once again. This could be written as a commutative diagram if one wishes. Such a lifting is called a spin structure. The only topological obstruction to a spin structure is the second Stiefel-Whitney class of the tangent bundle $w_2 \in H^2(\mathcal{M}^4;\mathbb{Z}_2)$. It is a basic exercise in obstruction theory to show that $M^4$ admits a spin structure if and only if $w_2$ vanishes. 
    
    These numbers ($\tau$, $\chi$, $b^+$, $b^-$,...) will play a crucial role in providing obstructions to the existence of gradient shrinking Soliton metrics. These are global topological invariants of $\mathcal{M}^4$, despite the fact that the data provided by the soliton equation is local. Luckily, there are results which relate these global quantities to integrals over $\mathcal{M}^4$ of local quantities. We briefly describe the structure responsible for these integral formulas.

    On a given oriented compact $(\mathcal{M}^4, g)$, one has the curvature operator $\mathcal{R}: \Lambda^2 \rightarrow \Lambda^2$. With respect to the decomposition into self-dual/anti-self-dual 2-forms, $\mathcal{R}$ admits the following decomposition:

    $$\mathcal{R} = \begin{pmatrix}
W^+ + \frac{R}{12} & \mathring{Rc} \\
\mathring{Rc} & W^- + \frac{R}{12} 
\end{pmatrix}$$

where $R$ is the scalar curvature, $\mathring{Rc}$ is the traceless Ricci tensor, and $W^{\pm}$ are the restrictions of the Weyl tensor's action $W: \Lambda^2 \rightarrow \Lambda^2$ to $\Lambda^\pm$ respectively. In coordinates, the actions are given by:

$$W^\pm(\phi) = W^\pm_{abcd}\phi_{cd}$$
$$R\phi = R\phi_{cd}$$
$$\mathring{Rc}(\phi) = \mathring{Rc}_{ac}\phi^{c}_{b} - \mathring{Rc}_{bc}\phi^{c}_{a}$$

restricted appropriately to the relevant subspaces suggested by the block decomposition of the matrix representing $\mathcal{R}$. Clever use of this decomposition yields the 4-dimensional Chern-Gauss-Bonnet formula and the Hirzebruch signature formula, which give integrals relating local quantities to global topological invariants as desired ($\mathring{Rc}$ is the traceless Ricci tensor of $g$) \cite{ernani}:

$$\chi(\mathcal{M}^4) = \frac{1}{8\pi^2}\int_{\mathcal{M}} |W^+|^2 + |W^-|^2 + \frac{R^2}{24} - \frac{|\mathring{Rc}|^2}{2}dV_g$$

$$\tau({\mathcal{M}^4}) = \frac{1}{12\pi^2}\int_{\mathcal{M}} |W^+|^2 - |W^-|^2 dV_g$$

Of great interest in the present work is the following:

$$2\chi \pm 3|\tau| = \frac{1}{4\pi^2}\int_{\mathcal{M}} 2|W^\pm|^2 + \frac{R^2}{24} - \frac{|\mathring{Rc}|^2}{2}dV_g$$

Our next section is devoted to applying this particular integral to provide topological obstructions to the existence of compact gradient soliton metrics. 

\section{The Hitchin-Thorpe Inequality}

    Let $\mathcal{M}^4$ be a compact oriented four-manifold. The subject of this section is the Hithcin-Thorpe inequality, which is an inequality involving oriented homotopy invariants on $\mathcal{M}^4$, namely $2\chi \pm 3|\tau| \geq 0$. Our first order of business is to demonstrate that this inequality provides an obstruction to the existence of an Einstein metric on $\mathcal{M}^4$. 

    \begin{thm}
    Let $(\mathcal{M}^4, g)$ be a compact oriented four-dimensional Einstein manifold. Then $\mathcal{M}^4$ satisfies the Hitchin-Thorpe inequality. 
    \end{thm}

    \begin{proof}

    Given that $(\mathcal{M}^4, g)$ is an Einstein manifold, by definition there exists $\rho \in \mathbb{R}$ such that 

    $$Rc_g = \rho g$$

    We demonstrate that the traceless Ricci tensor of $g$ vanishes. Recall that the traceless Ricci tensor $\overset{\circ}{Rc} = Rc - \frac{R}{4}g$, and the trace of the soliton equation yields $R = 4\rho$. Thus, we have

    $$\overset{\circ}{Rc} = \rho g - \frac{4\rho}{4}g = \rho(g - g) = 0.$$

    Now, using the local-to-global integral above:

    $$2\chi \pm 3|\tau| = \frac{1}{4\pi^2}\int_{\mathcal{M}} 2|W^\pm|^2 + \frac{R^2}{24} - \frac{|\mathring{Rc}|^2}{2}dV_g = \frac{1}{4\pi^2}\int_{\mathcal{M}} 2|W^\pm|^2 + \frac{R^2}{24}dV_g \geq 0$$

    and thus the Hitchin-Thorpe inequality holds on $(\mathcal{M}^4, g)$.

    \end{proof}

    Thus, the Hitchin-Thorpe inequality provides a topological obstruction to the existence of Einstein metrics on $\mathcal{M}^4$. In particular, the Hitchin-Thorpe inequality yields a topological obstruction to the existence of steady and expanding Ricci solitons on $\mathcal{M}^4$. Given that one would like a topological obstruction to the existence of soliton structures in general, Cao \cite{caoThree} proposed the following:

    \begin{quest}
    Must the Hitchin-Thorpe inequality also hold for compact gradient shrinking Ricci solitons $(\mathcal{M}^4, g)$? 
    \end{quest}

    This question remains unresolved in general, but there is partial progress in this direction in the form of the sufficient conditions which we shall prove (or state) below.
    
\begin{rem}
In all that follows - given scale invariance of the soliton equation - we normalize $\rho = \frac{1}{2}$. 
\end{rem}

    \begin{thm} (Klatt \cite{klatt})
    Let $(\mathcal{M}^4, g)$ be a compact spin gradient shrinking Ricci soliton. Then $\mathcal{M}^4$ satisfies the Hitchin-Thorpe inequality. 
    \end{thm}

Given that the topology of compact spin four-manifolds is well understood, this result is not hard to prove. 

\begin{proof}
Let $(\mathcal{M}^4, g)$ be a compact spin gradient shrinking Ricci soliton. In particular, Licherowicz \cite{lich} has demonstrated that metrics with $R > 0$ cannot exist when $\mathcal{M}^4$ is spin and the signature $\tau = -8\hat{\mathcal{A}} \neq 0$ where $\hat{\mathcal{A}}$ is the Hirzebruch $\hat{\mathcal{A}}$-genus. Given theorem 2.2, since $\mathcal{M}^4$ is spin it follows immediately that $\tau = 0$. Using Poincare duality, the expression of the Euler characteristic in terms of betti numbers, and the fact that $b_1 = b_3 \equiv 0$ for gradient shrinking Ricci solitons by theorem 2.3 one finds that $\chi = 2 + b_2 \geq 0$. Thus the Hitchin-Thorpe inequality is satisfied trivially. 
\end{proof}

Another sufficient condition is the following, first noted by Cao (private communication). This condition was also proved independently in \cite{aldir}. 

\begin{thm} Let $(\mathcal{M}^4, g)$ be a compact gradient shrinking Ricci soliton. Then the Hitchin-Thorpe inequality holds given the following sufficient condition: $$\int_{\mathcal{M}^4} R^2 dV_g \leq 6 Vol(\mathcal{M}^4, g)$$

\end{thm}

In order to prove this theorem, we begin with a lemma. 

\begin{lem}
    For a compact Gradient Shrinking Ricci Soliton $(\mathcal{M}^4, g, f)$, the following holds:

    $$\int_{\mathcal{M}} \frac{1}{2} R^2 - |Rc|^2 \,dV_g = Vol(\mathcal{M}, g) $$
\end{lem}

\begin{proof}
Consider the integral

$$ I = \int_{\mathcal{M}} \langle \frac{1}{2}g - Rc, Rc \rangle \,dV_g.$$

On one hand, direct evaluation yields the following:

$$\int_{\mathcal{M}} \langle \frac{1}{2}g - Rc, Rc \rangle \,dV_g = 
 \int_{\mathcal{M}} \frac{1}{2}R - |Rc|^2 \,dV_g = 
 Vol(\mathcal{M}, g) - \int_{\mathcal{M}} |Rc|^2  \,dV_g$$

On the other hand, using the soliton equation we may substitute:

$$\int_{\mathcal{M}} \langle \frac{1}{2}g - Rc, Rc \rangle \,dV_g = 
\int_{\mathcal{M}} \langle \nabla ^2 f, Rc \rangle \,dV_g$$

And using the divergence theorem, we evaluate 

$$I = - \int_{\mathcal{M}} \langle \nabla f, div(Rc) \rangle \,dV_g.$$

We may substitute further using the second contracted Bianchi identity:

$$I = - \frac{1}{2} \int_{\mathcal{M}} \langle \nabla f, \nabla R \rangle dV_g$$

Applying the divergence theorem again:

$$I = \frac{1}{2}\int_{\mathcal{M}} R \Delta f dV_g= \frac{1}{2}\int_{\mathcal{M}} R(2 - R) dV_g = \int_{\mathcal{M}} R dV_g - \frac{1}{2}\int_{\mathcal{M}} R^2 dV_g$$

Therefore, equating the two integrals:

$$I = 2Vol(\mathcal{M}, g) - \frac{1}{2}\int_{\mathcal{M}} R^2 dV_g = Vol(\mathcal{M}, g) - \int_{\mathcal{M}} |Rc|^2 dV_g$$

and hence 

 $$\int_{\mathcal{M}} \frac{1}{2} R^2 - |Rc|^2 \,dV_g = Vol(\mathcal{M}, g) $$

\end{proof}

We now proceed to prove the theorem. 

\begin{proof}

Suppose that the following holds on $(\mathcal{M}, g, f)$: 

$$\int_{\mathcal{M}} R^2 dV_g \leq 6Vol(\mathcal{M}, g).$$

Note that, upon substituting the traceless Ricci tensor with the ordinary one in the local-to-global theorem above one obtains:

$$2\chi(\mathcal{M}) \pm 3\tau(\mathcal{M}) = 
\frac{1}{8\pi^2} \int_{\mathcal{M}} 4|W^{\pm}|^2 - |Rc|^2 + \frac{1}{3}R^2 dV_g$$

In particular, 

$$\int_{\mathcal{M}} \frac{1}{3}R^2 dV_g - |Rc|^2 dV_g = 
\int_{\mathcal{M}} \frac{1}{2}R^2 dV_g - |Rc|^2 - \frac{1}{6}R^2 dV_g=
Vol(\mathcal{M}, g) - \frac{1}{6}\int_{\mathcal{M}} R^2 dV_g$$

Thus, by assumption:

$$2\chi(\mathcal{M}) - 3\tau(\mathcal{M}) =  \frac{1}{8\pi^2} (\int_{\mathcal{M}} 4|W^{-}|^2 dV_g + Vol(\mathcal{M},g) - \frac{1}{6}\int_{\mathcal{M}} R^2 dV_g) \geq 0$$

and 

$$2\chi(\mathcal{M}) + 3\tau(\mathcal{M}) =  \frac{1}{8\pi^2} (\int_{\mathcal{M}} 4|W^{+}|^2 dV_g + Vol(\mathcal{M},g) - \frac{1}{6}\int_{\mathcal{M}} R^2 dV_g) \geq 0$$

Therefore the Hitchin-Thorpe inequality is satisfied.

\end{proof}

In addition to these, Tadano \cite{tad} has given sufficient conditions in the form of diameter bounds involving the oscillation of the scalar curvature, and Cheng, Ribeiro Jr., and Zhou \cite{ernani} have furnished the following result:

\begin{thm}
If the potential function $f$ of a compact Gradient Shrinking Ricci Soliton $(\mathcal{M}^4, g, f)$ satisfies $f_{max} - f_{min} \leq log(5)$, then $\mathcal{M}^4$ satisfies the Hitchin-Thorpe inequality. 
\end{thm}
    
Given the present menagerie of sufficient conditions such as the ones stated above, one may inquire into the scope of these results. How many compact gradient shrinking Ricci solitons do these results apply to? The range of applicability of theorem 4.3 is already known well, since we recall that a smooth oriented four-manifold $\mathcal{M}^4$ is spin if and only if its second Stiefel-Whitney class vanishes. We begin to catch a glimpse of the efficacy of theorems 4.4 and 4.6 using results from $\emph{conformal geometry}$. We will see that these theorems in fact describe essentially the same restricted class of compact gradient shrinking Ricci solitons. Recall the Chern-Gauss-Bonnet formula for a Riemannian four-manifold (expressed in terms of the ordinary Ricci tensor using the definition of the traceless Ricci tensor):

$$8\pi^2 \chi(\mathcal{M}^4) = \int_{\mathcal{M}^4} |W|^2 dV_g + \int_{\mathcal{M}^4} \frac{1}{6}(R^2 - 3|Rc|^2) dV_g $$

The second integral is a conformal invariant in dimension four due to the conformal invariance of the first integral and the obvious conformal invariance of the left hand side. Following Chang \cite{chang}, we denote the integrand

$$\sigma_2(A_g) \equiv \frac{1}{6}(R^2 - 3|Rc|^2)$$ as it can be realized as the second symmetric polynomial acting on the eigenvalues of the Schouten tensor. Thus, the Chern-Gauss-Bonnet formula reads:


$$8\pi^2 \chi(\mathcal{M}) = \int_{\mathcal{M}} |W|^2 dV_g + \int_{\mathcal{M}^4} \sigma_2 dV_g$$

Denote the Yamabe invariant of a conformal class of metrics $[g]$ by $\mathcal{Y}(M, [g])$. We now define the following class of metrics with a view towards finding restrictions on metrics which satisfy the hypotheses of theorems 4.4 and 4.6.

\begin{defn}
Define the class $\mathcal{A} = \{g : \mathcal{Y}(M, [g]) > 0, \int_{\mathcal{M}} \sigma_2(A_g) dV_g > 0\}$
\end{defn}

The following analysis of metrics obeying the sufficient conditions in theorems 4.4 or 4.6 is split into two cases, the first case being that the case of strict inequality in both theorems and the second being the rigid case of equality in both theorems. We begin with the strict case. 

\begin{prop}
A normalized gradient shrinking Ricci soliton $(\mathcal{M}, g, f)$ satisfies

$$\int_{\mathcal{M}} R^2 dV_g < 6 Vol(\mathcal{M}, g).$$

if and only if $g \in \mathcal{A}$. 

\end{prop}

\begin{proof}

Note first that by theorem 2.2, all shrinking gradient Ricci solitons have positive scalar curvature, and thus $\mathcal{Y}(M, [g]) > 0$. Now compute:

$$\int_{\mathcal{M}} \sigma_2(A_g) dV_g > 0 \iff  
\int_{\mathcal{M}} \frac{1}{6}(R^2 - 3|Rc|^2) dV_g > 0 \iff  \int_{\mathcal{M}} \frac{1}{3}R^2 - |Rc|^2 dV_g > 0$$
And thus by lemma 4.5
$$\iff Vol(\mathcal{M}, g) -  \int_{\mathcal{M}} \frac{1}{6}R^2 dV_g > 0 \iff \int_{\mathcal{M}} R^2 dV_g < 6Vol(\mathcal{M}, g).$$

This completes the proof.

\end{proof}

Thus, metrics satisfying the condition of theorem 4.4 strictly are in $\mathcal{A}$, and vice versa. As for theorem 4.6, Cheng, Ribeiro Jr., and Zhou \cite{ernani} use the following lemma to derive the result of theorem 4.6.

\begin{lem}

Let $(\mathcal{M}^4, g, f)$ be a four-dimensional compact gradient shrinking Ricci soliton with $\rho$ normalized to $\frac{1}{2}$. Then the following holds:

$$8\pi^2 \chi(\mathcal{M}^4) \geq \int_{\mathcal{M}^4} |W|^2 dV_g + \frac{1}{24}Vol(\mathcal{M}^4, g)(5 - e^{f_{max} - f_{min}})$$

There is equality if and only if $(\mathcal{M}^4, g)$ is Einstein. 

\end{lem}

Using this lemma, we may find that metrics satisfying the hypotheses of theorem 4.6 strictly also lie in $\mathcal{A}$.

\begin{thm}

Metrics satisfying the hypotheses of theorem 4.6 strictly, i,e. $f_{max} - f_{min} < log(5)$, are in $\mathcal{A}$.

\end{thm}

\begin{proof}

Recalling the Chern-Gauss-Bonnet formula from above as well as lemma 4.11, we have that

$$\int_{\mathcal{M}^4} |W|^2 dV_g + \int_{\mathcal{M}^4} \sigma_2 dV_g = 8\pi^2 \chi(\mathcal{M}) \geq \int_{\mathcal{M}^4} |W|^2 dV_g + \frac{1}{24}Vol(\mathcal{M}^4, g)(5 - e^{f_{max} - f_{min}})$$

and therefore:

$$\int_{\mathcal{M}^4} \sigma_2 dV_g \geq \frac{1}{24}Vol(\mathcal{M}^4, g)(5 - e^{f_{max} - f_{min}}) > 0$$

by assumption. By theorem 2.2, $\mathcal{Y}(M, [g]) > 0$, and thus $g \in \mathcal{A}$ by definition. 

\end{proof}

The following corollary results

\begin{cor}
$\emph{All}$ metrics satisfying the hypotheses of theorem 4.4 strictly or theorem 4.6 strictly are in $\mathcal{A}$. 
\end{cor}

In fact, we can say more, since strict satisfaction of the hypothesis in theorem 4.4 coincides precisely with being in $\mathcal{A}$. 
\begin{thm}
Let $(\mathcal{M}^4, g, f)$ be a compact gradient shrinking Ricci soliton satisfying $f_{max} - f_{min} < log(5)$, then 

$$\int_{\mathcal{M}^4}R^2 dV_g < 6 Vol(\mathcal{M}^4, g)$$

\end{thm}

\begin{proof}
By theorem 4.10, $g \in \mathcal{A}$, and thus by proposition 4.8 $g$ satisfies the stated condition. 

\end{proof}
Thus, the strict satisfaction of the conditions in theorem 4.6 entails the strict satisfaction of the condition in theorem 4.4, and thus the assumption of theorem 4.4 is weaker than that of theorem 4.6 in the strict case.

 How restrictive is the condition that $g \in \mathcal{A}$? Among other properties of metrics in $\mathcal{A}$, such as our ability to prescribe $\sigma_2$ as any positive smooth function within a conformal class, we have the following result \cite{chang}

\begin{thm}
Metrics in $\mathcal{A}$ have positive Ricci curvature. 
\end{thm}
Thus we have the following theorem restricting the geometry of metrics satisfying these two sufficient conditions:
\begin{thm}
Metrics satisfying the hypotheses of theorem 4.4 strictly or theorem 4.6 strictly have positive Ricci curvature.
\end{thm}
\begin{proof}

Use theorem 4.13 and corollary 4.11.

\end{proof}

 Given that Cao \cite{caoTwo} has constructed a countable class of gradient shrinking Ricci solitons with Ricci curvature not positive everywhere on twisted projective line bundles over $\mathbb{CP}^1$, the strict sufficient conditions of theorems 4.4 and 4.6 already "miss" many examples of compact gradient shrinking Ricci solitons and are thus proved not to be universally applicable.

The discussion above only applies to the case of strict inequality in both sufficient conditions. We now proceed to examine to what metrics the rigid case of equality applies in either theorem. We first examine the rigid case of theorem 4.4.

\begin{thm}
A compact gradient shrinking Ricci soliton $(\mathcal{M}^4, g, f)$ satisfies

$$\int_{\mathcal{M}^4} R^2 dV_g = 6 Vol(\mathcal{M}^4, g)$$

if and only if

$$\int_{\mathcal{M}^4} \sigma_2 dV_g = 0$$

\end{thm}

\begin{proof}
This proof is essentially the proof of proposition 4.8 in the case that the strict inequality is changed to an equality. As such, it is left to the reader to verify. 
\end{proof}

The rigid case of theorem 4.6 may be obtained from lemma 4.9.

\begin{thm}
A compact gradient shrinking Ricci soliton $(\mathcal{M}^4, g, f)$ satisfying $f_{max} - f_{min} = log(5)$ must satisfy

$$\int_{\mathcal{M}^4} \sigma_2 dV_g \geq 0$$

\end{thm}

\begin{proof}
By lemma 4.9 one has the following:

$$\int_{\mathcal{M}^4} |W|^2 dV_g + \int_{\mathcal{M}^4} \sigma_2 dV_g = 8\pi^2 \chi(\mathcal{M}) \geq \int_{\mathcal{M}^4} |W|^2 dV_g + \frac{1}{24}Vol(\mathcal{M}^4, g)(5 - e^{f_{max} - f_{min}})$$

And, using the assumed condition one has:

$$\int_{\mathcal{M}^4} |W|^2 dV_g + \int_{\mathcal{M}^4} \sigma_2 dV_g = 8\pi^2 \chi(\mathcal{M}) \geq \int_{\mathcal{M}^4} |W|^2 dV_g$$

Subtracting concludes the proof:

$$\int_{\mathcal{M}^4} \sigma_2 dV_g \geq 0$$

\end{proof}

Thus, in the case of equality for both theorems, the sufficient condition of theorem 4.6 appears to be more widely applicable than the sufficient condition of theorem 4.4, as it may include metrics in $\mathcal{A}$. However we may see in fact that theorem 4.6 actually entails theorem 4.4, demonstrating that the assumption of theorem 4.6 is a stronger assumption than that of theorem 4.4.

\begin{thm}
A compact gradient shrinking Ricci soliton $(\mathcal{M}^4, g, f)$ satisfying $f_{max} - f_{min} \leq log(5)$ must satisfy

$$\int_{\mathcal{M}^4} R^2 dV_g \leq 6 Vol(\mathcal{M}^4, g).$$

\end{thm}

\begin{proof}
If $(\mathcal{M}^4, g, f)$ satisfies $f_{max} - f_{min} < log(5)$, then $g \in \mathcal{A}$ and thus by theorem 4.12 satisfies

$$\int_{\mathcal{M}^4} R^2 dV_g < 6 Vol(\mathcal{M}^4, g).$$

If $(\mathcal{M}^4, g, f)$ satisfies $f_{max} - f_{min} = log(5)$, then by theorem 4.16 $g \in \mathcal{A}$ again or 

$$\int_{\mathcal{M}^4} \sigma_2 dV_g = 0$$

in which case 

$$\int_{\mathcal{M}^4} R^2 dV_g = 6 Vol(\mathcal{M}^4, g)$$

by theorem 4.15.
\end{proof}

Thus the hypothesis of theorem 4.4 is a necessary condition for metrics satisfying the hypothesis of theorem 4.6. As far as the efficacy of theorem 4.4 goes, it applies only to metrics in $\mathcal{A}$ - in particular metrics of positive Ricci curvature - or to metrics for which the integral conformal invariant is precisely $0$, which is equivalent to the case of equality in theorem 4.4. Both of these conditions appear quite restrictive. 

To that end, a brief meta-mathematical detour is in order. It is easy to show using the local-to-global expression of section 3 and lemma 4.5 that the following inequality is equivalent to the Hitchin-Thorpe inequality on compact gradient shrinking Ricci solitons:

$$\int_{\mathcal{M}^4} 24 |W^\pm|^2 - R^2 + 6 dV_g \geq 0.$$

    From this vantage, it seems that control of $|W^\pm|^2$ in terms of $R^2$ is necessary to understand the behaviour of the Hitchin-Thorpe inequality on gradient shrinking Ricci solitons. Current work on this question - as in the case of the sufficient conditions above - treats the $|W^\pm|^2$ term as one that can be thrown away. Perhaps an understanding of this term is crucial to understanding the general situation. 

\section{Topological Obstructions to the Existence of Compact Kähler Gradient Shrinking Ricci Solitons}

The question of finding topological obstructions to the existence of gradient shrinking Ricci solitons is difficult in general. In this section we now proceed to investigate the question of topological obstructions to the existence of compact Kähler gradient shrinking Ricci solitons and find that this additional structure allows for a very nice description of these spaces. This section takes the form of an outline of the subject and may be used as a literature guide. Let $X^2$ be a compact complex surface with the natural orientation.

\begin{defn}
A four-dimensional $\emph{compact Kähler gradient shrinking Ricci soliton}$ is the compact complex manifold $X^2$ together with a compatible Kahler metric $g_{i\bar{j}}$ satisfying

$$R_{i\bar{j}} + \nabla_i \bar{\nabla}_{j} f = \rho g_{i\bar{j}}$$

for some $f \in C^\infty(\mathcal{M})$ and $\rho > 0$. 

\end{defn}

We begin our investigation of these spaces by noting, in light of the last section, that we have the following:

\begin{thm}
A compact Kähler gradient shrinking Ricci soliton $(X^2, g, f)$ must satisfy the strict inequality $2\chi + 3\tau > 0$. 

\end{thm}

\begin{proof}
By the work of Derdzinski \cite{derd}, $24|W^+|^2 = R^2$ on Kähler $X^2$. Thus, the local to global theorem of section 3 and lemma 4.5 demonstrate the theorem trivially. 
\end{proof}

As far as the last section is concerned, this is a powerful result. However, the Kähler condition guarantees $\emph{more}$ topological information than this obstructions does. In fact, it yields a complete $\emph{classification}$ of compact Kähler gradient shrinking Ricci solitons! 

\begin{thm}
The only compact Kähler gradient shrinking Ricci solitons are the Kahler-Einstein metrics above and the following two non-Einstein cases: the Koiso-Cao soliton on $\mathbb{CP}^2 \# -\mathbb{CP}^2$, or the Wang-Zhu soliton on $\mathbb{CP}^2 \# -2\mathbb{CP}^2$
\end{thm}

We begin outlining the proof of this result with a lemma regarding the so-called first $\emph{Chern class}$ of $X^2$. The first Chern class of $X^2$ is the first Chern class of the anti-canonical line bundle $K^{-1}$. The first Chern class of a line bundle $L$ is a cohomology class $c_1(L) \in H^2(X^2;\mathbb{R})$ which can be associated to $L$ in multiple ways - in particular via Chern-Weil theory, a classifying space construction, or Čech cohomology \cite{wells}. $K^{-1}$ is defined using the fact that isomorphism classes of complex line bundles over a complex manifold form a group - called the $\emph{Picard group}$ of $X^2$ - under the tensor product. $K^{-1}$ is simply the inverse of the canonical line bundle of top forms $K \equiv \Omega_X^2$ in this group. In the Kähler setting, a concrete geometric way to view $c_1(X^2)$ is as the cohomology class of the Ricci form, and indeed this is the interpretation we will use. 

\begin{lem}
A compact Kähler gradient shrinking Ricci soliton must have positive first Chern class, meaning that it may be represented by a Kähler form. 
\end{lem}

\begin{proof}
Given that the first Chern class of a Kähler manifold is the cohomology class of the Ricci form, we may use the soliton equation to yield the following:

$$[Rc] = [\rho \omega - \partial \bar{\partial} f] = \rho [\omega]$$

Since $c_1(X^2) = [Rc]$ may be represented by a Kähler class, it is positive by definition. 

\end{proof}

This lemma tells us that compact Kähler gradient shrinking Ricci solitons are in fact examples of familiar spaces. We recall some basic facts from algebraic geometry regarding $\emph{del-Pezzo surfaces}$. 

\begin{defn}

A compact complex surface $X^2$ is a $\emph{del-Pezzo surface}$ if its anti-canonical bundle $K^{-1}$ is ample.

\end{defn}

  Intuitively speaking, a line bundle $L$ is $\emph{ample}$ if one can take powers of it to obtain a line bundle $L'$ such that $X^2$ embeds in some projective space such that the restriction of the tautological bundle on that projective space to the image of $X^2$ is $L'$. For our purposes, it suffices to consider the Kodaira Embedding Theorem \cite{huy}:

 \begin{thm}
Let $X^2$ be a compact Kähler manifold. A line bundle $L$ on $X^2$ is positive if and only if $L$ is ample. In this case, $X^2$ is projective.
 \end{thm}

 We may now recognize compact Kähler gradient shrinking Ricci solitons $X^2$ as del-Pezzo surfaces, since the first Chern class of the anti-canonical bundle is positive, and thus the anti-canonical bundle is ample. Recognizing compact Kähler gradient shrinking Ricci solitons as del-Pezzo surfaces is extremely powerful, as it is a classical result of algebraic geometry that these are classified:

 \begin{thm}
The only del-Pezzo surfaces are $\mathbb{CP}^2 \# -k\mathbb{CP}^2$ for $0 \leq k \leq 8$ and $S^2 \times S^2$. 
 \end{thm}

For the sake of continuing the classification, we now note what is known in the Einstein case. As in section $2$, compact Kähler gradient shrinking Ricci solitons are generalizations of the sister notion of positive $\emph{Kähler Einstein}$ manifolds, who satisfy

$$R_{i\bar{j}} = \rho g_{i\bar{j}}$$

for some $\rho > 0$. These manifolds have been completely classified by Tian \cite{tianOne}

\begin{thm}
A compact complex surface $(X^2, J)$ admits a compatible Kähler-Einstein metric with R $>$ 0 if and only if its anti-canonical line bundle $K^{-1}$ is ample and its Lie algebra of holomorphic vector fields is reductive. 
\end{thm}

This result restricts the spaces which admit such metrics to the following: $\mathbb{CP}^2$, $S^2 \times S^2$, $\mathbb{CP}^2 \# -k(\mathbb{CP}^2)$, $3 \leq k \leq 8$. 

 A uniqueness result due to Tian \cite{tianTwo} \cite{tianThree} now comes in handy:

\begin{thm}
There exists at most one Kähler-Ricci soliton on any compact Kähler manifold with positive first Chern class modulo holomorphic automorphisms.
\end{thm}

Therefore, the Kähler-Einstein metrics above are all unique. Furthermore, the Koiso-Cao soliton on $\mathbb{CP}^2 \# -\mathbb{CP}^2$, and the Wang-Zhu soliton on $\mathbb{CP}^2 \# -2\mathbb{CP}^2$ exist and by theorem 5.11 are unique as well.  Thus, up to holomorphic automorphisms, the only compact Kähler Ricci solitons in dimension four are the Einstein manifolds classified by Tian, the Koiso-Cao soliton on $\mathbb{CP}^2 \# -\mathbb{CP}^2$, and the Wang-Zhu soliton on $\mathbb{CP}^2 \# -2\mathbb{CP}^2$ as claimed, and one has complete knowledge of the topology of compact Kähler gradient shrinking Ricci solitons.

One might ask if these results could be extended to include more general structures, e.g. Symplectic. While a complete classification is beyond the scope of this work, an obstruction in this more general case follows from the powerful work of Taubes \cite{taubes} which utilizes Seiberg-Witten theory.

\begin{thm} A symplectic four-manifold with $b_+ > 1$ admits no metrics of positive scalar curvature. 
\end{thm}

Thus given theorem 2.2 we have the immediate corollary:

\begin{cor}
A compact symplectic gradient shrinking Ricci soliton $\mathcal{M}^4$ must have $b_+ \leq 1$. 
\end{cor}

This is an example of a step towards finding topological obstructions to the existence of compact gradient shrinking Ricci solitons in more general settings where additional structures exist.

\end{document}